\documentclass[12pt]{amsart}
\usepackage[margin=0.95in]{geometry}

\usepackage{enumitem}
\usepackage{amssymb, amsmath}
\usepackage{amsthm}
\usepackage{hyperref}

\usepackage{cases}

\usepackage{tikz}
\usetikzlibrary{decorations.markings, decorations.pathreplacing, positioning,arrows, matrix, decorations.pathmorphing, shapes.geometric, calc}
\usetikzlibrary{backgrounds, decorations}

\newtheorem{theorem}{Theorem}[section]
\newtheorem{lemma}[theorem]{Lemma}
\newtheorem{corollary}[theorem]{Corollary}

\newtheorem{conjecture}[theorem]{Conjecture}
\newtheorem{example}[theorem]{Example}
\newtheorem{proposition}[theorem]{Proposition}
\newtheorem{thmletter}{Theorem}

\newtheorem{corolletter}[thmletter]{Corollary}

\newcommand{\p}[1]{\noindent {\newline\bf #1.}}

\newcommand{\aut}{\operatorname{Aut}}

\newcommand{\core}{\operatorname{Core}}
\newcommand{\PCP}{$\operatorname{PCP_{FM}}$}
\newcommand{\PCPG}{$\operatorname{PCP_{FG}}$}
\newcommand{\EPG}{$\operatorname{AEP_{FG}}$}
\newcommand{\EPM}{$\operatorname{AEP_{FM}}$}
\newcommand{\rk}{\operatorname{rk}}
\newcommand{\eq}{\operatorname{Eq}}
\newcommand{\im}{\operatorname{Image}}
\newcommand{\fix}{\operatorname{Fix}}

\title[PCP and equalisers for certain morphisms]{The Post Correspondence Problem and equalisers for certain free group and monoid morphisms}

\author{Laura Ciobanu}
\address{ 
 Heriot-Watt University,   Edinburgh EH14 4AS,
 Scotland}
 \email{l.ciobanu@hw.ac.uk}

\author{Alan D. Logan}
\address{
Heriot-Watt University,   Edinburgh EH14 4AS,
 Scotland}
 \email{a.logan@hw.ac.uk}



\begin{document}

\maketitle

\begin{abstract}
A marked free monoid morphism is a morphism for which the image of each generator starts with a different letter, and immersions are the analogous maps in free groups. We show that the (simultaneous) PCP is decidable for immersions of free groups, and provide an algorithm to compute bases for the sets, called equalisers, on which the immersions take the same values. We also answer a question of Stallings about the rank of the equaliser.  

Analogous results are proven for marked morphisms of free monoids.\end{abstract}

\section{Introduction}

In this paper we prove results about the classical Post Correspondence Problem (\PCP{}), which we state in terms of equalisers of free monoid morphisms, and the analogue problem \PCPG{} for free groups (\cite{CMV}, \cite{Myasnikov2014Post}), and we describe the solutions to \PCP{} and \PCPG{} for certain classes of morphisms. While the classical \PCP{} is famously undecidable for arbitrary maps of free monoids \cite{Post1946Correspondence} (see also the survey \cite{Harju1997Morphisms} and the recent result of Neary \cite{Neary2015undecidability}),
\PCPG{} for free groups is an important open question \cite[Problem 5.1.4]{Dagstuhl2019}. Additionally, for both free monoids and free groups there are only few results describing algebraically the solutions to classes of instances known to have decidable \PCP{} or \PCPG{}. Our results apply to \emph{marked} morphisms in the monoid case, and to their counterparts in free groups, called  \emph{immersions}.
Marked morphisms are the key tool used in resolving the \PCP{} for the free monoid of rank two \cite{Ehrenfeucht1982generalized}, and therefore understanding the solutions to the \PCPG{} for immersions is an important step towards resolving the \PCPG{} for the free group of rank two.
The density of marked morphisms and immersions among all the free monoid or group maps is strictly positive (Section \ref{sec:Density}), so our results concern a significant proportion of instances.

An \emph{instance} of the \PCP{} is a tuple $I=(\Sigma, \Delta, g, h)$, where $\Sigma, \Delta$ are  finite alphabets, $\Sigma^*, \Delta^*$ are the respective free monoids, and $g, h: \Sigma^*\rightarrow \Delta^*$ are morphisms. The \emph{equaliser} of $g, h$ is
$\eq(g, h)=\{x\in \Sigma^* \mid g(x)=h(x)\}.$ The \PCP{} is the decision problem:

\noindent{\newline
$\pmb{\operatorname{PCP_{FM}}}$:
Given $I=(\Sigma, \Delta, g, h)$, is the equaliser $\eq(g, h)$ trivial?}
\newline

Analogously, an instance of the \PCPG{} is a four-tuple $I=(\Sigma, \Delta, g, h)$ with $g, h: F(\Sigma)\rightarrow F(\Delta)$ morphisms between the free groups $F(\Sigma)$ and $F(\Delta)$, and \PCPG{} is the decision problem pertaining to the similarly defined $\eq(g, h)$ in free groups. 

Beyond \PCP{}, in this paper we also consider the \emph{Algorithmic Equaliser Problem}, denoted \EPM{} (or \EPG{} in the group case), which for an instance $I=(\Sigma, \Delta, g, h)$ with $g, h$ free monoid morphisms (or free group morphisms for \EPG{}), says:

\noindent{\newline
$\pmb{\operatorname{AEP_{FM}}}$:
Given $I=(\Sigma, \Delta, g, h)$, output}
\begin{enumerate}[label={(\alph*)}]
\item\label{AEP:Pta} a finite basis for $\eq(g, h)$, or
\item\label{AEP:Ptb} a finite automaton recognising the set $\eq(g, h)$.
\end{enumerate}

If a finite basis or finite automaton for $\eq(g, h)$ does not exist then Part (a) or (b), respectively, of the problem is {insoluble}.
Note that (a) and (b) are connected: for free groups these two problems are in fact the same when $\eq(g, h)$ is finitely generated, while for free monoids (\ref{AEP:Pta}) implies (\ref{AEP:Ptb}).
Part (\ref{AEP:Pta}) of the \EPM{} is known to be soluble when $|\Sigma|=2$ and one of $g$ or $h$ is non-periodic, and insoluble otherwise \cite{holub2003binary} \cite[Corollary 6]{Harju1997Morphisms}.

\p{Sets of morphisms}
We are particularly interested in sets $S$ of morphisms (not just two morphisms $f$, $g$) and their equalisers $\eq({S})=\bigcap_{g, h\in S}\eq(g, h)$, and we prove structural results for arbitrary sets and algorithmic results for finite sets.
Our results resolve the \emph{simultaneous \PCPG{} and \PCP{}} for immersions and marked morphisms; these problems take as input a finite set $S$ of maps and ask the same questions about equalisers as in the classical setting.
Analogously, one could further define the ``simultaneous \EPG{} and \EPM{}''. However, the simultaneous \EPG{} is equivalent to the \EPG{}, and Part (\ref{AEP:Ptb}) of the simultaneous \EPM{} is equivalent to Part (\ref{AEP:Ptb}) of the \EPM{}, as follows.
As bases of intersections of finitely generated subgroups of free groups are computable (and as Parts (\ref{AEP:Pta}) and (\ref{AEP:Ptb}) of the \EPG{} are equivalent), if the \EPG{} is soluble for a class $\mathcal{C}$ of maps then there exists an algorithm with input a finite set ${S}$ of morphisms from $F(\Sigma)$ to $F(\Delta)$, $S \subseteq \mathcal{C}$, and output a basis for $\eq({S})$.
Similarly, automata accepting intersections of regular languages are computable, and so if Part (\ref{AEP:Ptb}) of the \EPM{} is soluble for a class $\mathcal{C}$ of maps then there exists an algorithm with input a finite set ${S}$ of morphisms from $\Sigma^*$ to $\Delta^*$, $S \subseteq \mathcal{C}$, and output a finite automaton whose language is $\eq({S})=\bigcap_{g, h\in S}\eq(g, h)$.

\p{Main results}
A set of words $\mathbf{s}\subseteq\Delta^{\ast}$ is \emph{marked} if any two distinct $u, v\in\mathbf{s}$ start with a different letter of $\Delta$, which implies $|\mathbf{s}|\leq |\Delta|$. A free monoid morphism ${f}: \Sigma^*\rightarrow\Delta^*$ is \emph{marked} if the set ${f}(\Sigma)$ is marked.
An \emph{immersion of free groups} is a morphism $f: F(\Sigma)\rightarrow F(\Delta)$ where the set $f(\Sigma\cup\Sigma^{-1})$ is marked (see Section \ref{sec:immersions} for equivalent formulations). Halava, Hirvensalo and de Wolf \cite{Halava2001marked} showed that \PCP{} is decidable for marked morphisms; inspired by their methods we were able to obtain stronger results (Theorem \ref{thm:mainMonoids}) for this kind of map, as well as expand to the world of free groups (Theorem \ref{thm:main}), where we employ `finite state automata'-like objects called \emph{Stallings graphs}.
\begin{thmletter}
\label{thm:mainMonoids}
If ${S}$ is a set of marked morphisms from $\Sigma^*$ to $\Delta^*$, then there exists a finite alphabet $\Sigma_{{S}}$ and a marked morphism ${\psi}_{{S}}:\Sigma_{{S}}^*\rightarrow \Sigma^*$ such that $\im(\psi_{{S}})=\eq({S})$. Moreover, for ${S}$ finite, there exists an algorithm with input ${S}$ and output the marked morphism ${\psi}_{{S}}$.
\end{thmletter}

\begin{corolletter}
\label{corol:SPCPMonoid}
The simultaneous \PCP{} is decidable for marked morphisms of free monoids.
\end{corolletter}

\begin{thmletter}
\label{thm:main}
If $S$ is a set of immersions from $F(\Sigma)$ to $F(\Delta)$, then there exists a finite alphabet $\Sigma_S$ and an immersion $\psi_S:F(\Sigma_S)\rightarrow F(\Sigma)$ such that $\im(\psi_S)=\eq(S)$. Moreover, when $S$ is finite, there exists an algorithm with input $S$ and output the immersion $\psi_S$.
\end{thmletter}

\begin{corolletter}
\label{corol:SPCP}
The simultaneous \PCPG{} is decidable for immersions of free groups.
\end{corolletter}

\p{The Equaliser Conjecture}
Our work was partially motivated by Stallings' Equaliser Conjecture for free groups, which dates from 1984 \cite[Problems P1 \& 5]{Stallings1987Graphical} (also \cite[Problem 6]{Dicks1996Group} \cite[Conjecture 8.3]{Ventura2002Fixed} \cite[Problem F31]{Baumslag2002Open}). Here $\rk(H)$ stands for the \emph{rank}, or minimum number of generators, of a subgroup $H$:
\begin{conjecture}[The Equalizer Conjecture, 1984]
\label{Qn:StallingsRank}
If $g, h: F(\Sigma)\rightarrow F(\Delta)$ are injective morphisms then $\rk(\eq(g, h))\leq|\Sigma|$.
\end{conjecture}
 
This conjecture has its roots in ``fixed subgroups'' $\fix(\phi)$ of free group endomorphisms $\phi: F(\Sigma)\rightarrow F(\Sigma)$ (if $\Sigma=\Delta$ then $\fix(\phi)=\eq(\phi, \operatorname{id})$), where Bestvina and Handel proved that $\rk(\fix(\phi))\leq|\Sigma|$ for $\phi$ an automorphism \cite{Bestvina1992Traintracks}, and Imrich and Turner extended this bound to all endomorphisms \cite{Imrich1989Endomorphisms}. Bergman further extended this bound to all sets of endomorphisms \cite{bergman1999supports}.
Like Bergman's result, our first corollary of Theorem \ref{thm:main} considers sets of immersions, which are injective, and answers Conjecture \ref{Qn:StallingsRank} for immersions.
\begin{corolletter}
\label{thm:rankSETS}
If $S$ is a set of immersions from $F(\Sigma)$ to $F(\Delta)$ then $\rk(\eq(S))\leq|\Sigma|$.
\end{corolletter}

In free monoids, equalisers of injections are free \cite[Corollary 4]{Harju1997Morphisms} but they are not necessarily regular languages (and hence not necessarily finitely generated) \cite[Example 6]{Harju1997Morphisms}. In order to understand equalisers $\eq(S)$ of sets of maps we need to understand intersections in free monoids. Recall that the intersection $A^*\cap B^*$ of two finitely generated free submonoids of a free monoid $\Sigma^*$ is free \cite{tilson1972intersection} and one can find a regular expression that represents a basis of $A^*\cap B^*$ \cite{blattner1977automata}. However, the intersection is not necessarily finitely generated \cite{karhumaki1984note}. The following result is surprising because we have finite generation, even for the intersection $\eq(S)=\bigcap_{g, h\in S}\eq(g, h)$. 

\begin{corolletter}
\label{corol:rankMonoid}
If ${S}$ is a set of marked morphisms from $\Sigma^*$ to $\Delta^*$ then $\eq({S})$ is a free monoid with $\rk(\eq({S}))\leq|\Sigma|$.
\end{corolletter}

\p{The Algorithmic Equaliser Problem}
The \EPG{} is insoluble in general, as equalisers in free groups are not necessarily finitely generated \cite[Section 3]{Ventura2002Fixed}, and is an open problem of Stallings' if both maps are injective \cite[Problems P3 \& 5]{Stallings1987Graphical}. Our next corollary of Theorem \ref{thm:main} resolves this open problem for immersions.

\begin{corolletter}
\label{thm:basis}
The \EPG{} is soluble for immersions of free groups.
\end{corolletter}

The \EPM{} is insoluble in general, primarily as equalisers are not necessarily regular languages \cite[Example 4.6]{ehrenfeucht1978elementary}. Even for maps whose equalisers form regular languages, the problem remains insoluble  \cite{saarela2010noneffective}.
Another corollary of Theorem \ref{thm:mainMonoids} is the following.

\begin{corolletter}
\label{corol:basisMonoid}
The \EPM{} is soluble for marked morphisms of free monoids.
\end{corolletter}

\p{Outline of the article}
In Section \ref{sec:PCPM} we prove Theorem \ref{thm:mainMonoids} and its corollaries.
The remainder of the paper focuses on free groups, where the central result is Theorem \ref{thm:main1}, which is Theorem \ref{thm:main} for $|S|=2$.
In Section \ref{sec:immersions} we reformulate immersions in terms of Stallings' graphs.
In Section \ref{sec:CoreGraph} we define the ``reduction'' $I'=(\Sigma', \Delta', g', h')$ of an instance $I=(\Sigma, \Delta, g, h)$ of the \EPG{} for immersions. Repeatedly computing reductions is the key process in our algorithm.
In Section \ref{sec:prefix} we prove the process of reduction reduces the ``prefix complexity'' of an instance (so the word ``reduction'' makes sense).
In Section \ref{sec:TheoremsForPairs} we prove Theorem \ref{thm:main1}, mentioned above.
In Section \ref{sec:MainResults} we prove Theorem \ref{thm:main} and its corollaries.
In Section \ref{sec:Complexity} we give a complexity analysis for both our free monoid and free group algorithms, and in Section \ref{sec:Density} we show that the density of marked morphisms and immersions among all the free monoid or group maps is strictly positive.

\section*{Acknowledgements}
The authors were supported by EPSRC Standard Grant EP/R035814/1. The first-named author would like to thank the organisers of the Dagstuhl seminar 19131 \emph{Algorithmic Problems in Group Theory}, where the topics addressed in this paper were discussed and listed as important open questions in the theory of free groups \cite[Direction 5.1.4]{Dagstuhl2019}.

\section{Marked morphisms in free monoids}
\label{sec:PCPM}
In this section we prove Theorem \ref{thm:mainMonoids} and its corollaries. We use the following immediate fact. 
\begin{lemma}
\label{lem:MarkedInj}
Marked morphisms of free monoids are injective.
\end{lemma}

\begin{proof}
Let $f: \Sigma^*\rightarrow \Delta^*$ be marked and let $x\neq y$ be nontrivial. One can write $x=zax'$ and $y=zby'$, where $a, b\in\Sigma$ are the first letter where $x$ and $y$ differ. As $f$ is marked, $f(a)\neq f(b)$, hence $f(x)=f(z)f(a)f(x')\neq f(z)f(b)f(y')=f(y)$, so $f$ is injective.
\end{proof}

We may assume $\Sigma\subseteq \Delta$, as $|\Sigma|\leq |\Delta|$ holds whenever $f: \Sigma^*\rightarrow \Delta^*$ is marked. 

Consider morphisms $g: \Sigma_1^*\rightarrow \Delta^*$ and $h: \Sigma_2^*\rightarrow \Delta^*$. The set of non-empty words over an alphabet $\Sigma$ is denoted $\Sigma^+$. For $a \in \Delta$, a pair $(u,v) \in \Sigma_1^+ \times \Sigma_2^+$ is an \emph{$a$-block} if (i) $g(u)=h(v)$ starts with $a$, and (ii) $u$ and $v$ are minimal, that is, the length $|g(u)|=|h(v)|$ is minimal among all such pairs.
If the pair $(g, h)$ has blocks $a_i=(u_i, v_i)$, $1 \leq i \leq m$, then let $\Sigma'$ be the alphabet consisting of these blocks and define $g':(\Sigma')^* \mapsto \Sigma_1^*$ by $g'(a_i)=u_i$ and $h':(\Sigma')^* \mapsto \Sigma_2^*$ by $h'(a_i)=v_i$. These maps are computable and, by an identical logic to \cite[Section 2]{Halava2001marked}, are seen to be marked. Then $gg'=hh'$, and we let $k=gg'=hh'$ (so $k: (\Sigma')^*\rightarrow \Delta^*$). Since $k$ is the composition of marked morphisms, it is itself marked. We therefore have the following.

\begin{lemma}
\label{lem:unfoldableretractsMonoids}
If $g:\Sigma_1^*\rightarrow \Delta^*$ and $h:\Sigma_2^*\rightarrow \Delta^*$ are marked morphisms then the corresponding maps $g': \Sigma'^*\rightarrow \Sigma_1^*$, $h': \Sigma'^*\rightarrow \Sigma_2^*$ and $k: \Sigma^*\rightarrow \Delta^*$, $k=gg'=hh'$, are marked and are computable.
\end{lemma}


The \emph{reduction} of an instance $I=(\Sigma, \Delta, g, h)$ of the marked \PCP{}, as defined in \cite{Halava2001marked}, is the instance $I':=(\Sigma', \Delta, g', h')$ where $\Sigma'$ is defined as above, and where $g'$ and $h'$ are as above, but with codomain $\Delta$ (which we may do as $\Sigma\subseteq\Delta$).
We additionally assume that $\Sigma'\subseteq\Sigma$; we can do this as $|\Sigma'|\leq|\Sigma|$ by Lemma \ref{lem:unfoldableretractsMonoids}.

The following relies on \cite[Lemma 1]{Halava2001marked}, which we strengthen by replacing the notion of ``equivalence'' with that of ``strong equivalence'': Two instances $I_1$ and $I_2$ of the \PCP{} are \emph{strongly equivalent} if their equalisers are isomorphic, which we write as $\eq(I_1)\cong\eq(I_2)$.

\begin{lemma}
\label{lem:MarkedLemma1}
Let $I'=(\Sigma', \Delta', g', h')$ be the reduction of $I=(\Sigma, \Delta, {g}, {h})$ where ${g}$ and ${h}$ are marked. Then $I$ and $I'$ are strongly equivalent, and ${g}'(\eq(I'))=\eq(I)={h}'(\eq(I'))$.
\end{lemma}

\begin{proof}
Firstly, note that ${g}'(\eq(I'))\leq\eq(I)$ \cite[Lemma 1, paragraph 2]{Halava2001marked}. From \cite[Lemma 1, paragraph 1]{Halava2001marked} it follows that ${g}'(\eq(I'))\geq\eq(I)$, so ${g}'(\eq(I'))=\eq(I)$ . As ${g}'$ is injective, the map ${g}'|_{\eq(I')}$ is an isomorphism. Hence, $I$ and $I'$ are strongly equivalent, and, by symmetry for the $h'$ map, ${g}'(\eq(I'))=\eq(I)={h}'(\eq(I'))$ as required.
\end{proof}

We can now improve the existing result on the marked \PCP{}. We store a morphism $f:\Sigma^*\rightarrow \Delta^*$ as a list $(f(a))_{a\in \Sigma}$.

\begin{theorem}
\label{thm:Tweaking}
If $I=(\Sigma, \Delta, g, h)$ is an instance of the marked \PCP{} then there exists an alphabet $\Sigma_{g, h}$ and a marked morphism ${\psi}_{g, h}: \Sigma_{g, h}^*\rightarrow \Sigma^*$ such that $\im(\psi_{g, h})=\eq(I)$. Moreover, there exists an algorithm with input $I$ and output the marked morphism ${\psi}_{g, h}$.
\end{theorem}

\begin{proof}
We explain the algorithm, and note at the end that the output is a marked morphism ${\psi}_{g, h}: F(\Sigma_{g, h})\rightarrow F(\Sigma)$ with the required properties, and so the result follows. 

Begin by making reductions $I_0, I_1, I_2, \ldots$, starting with $I_0=I=(\Sigma, \Delta, {g}, {h})$, the input instance. Then by \cite[Section 5, paragraph 1]{Halava2001marked} we will obtain an instance $I_j=(\Sigma_j, \Delta, {g}_j, {h}_j)$ such that one of the following will occur:
\begin{enumerate}
\item\label{Tweaking:1} $|\Sigma_j|=1$.
\item\label{Tweaking:2} $|{g}_j(a)|=1=|{h}_j(a)|$ for all $a\in\Sigma_j$.
\item\label{Tweaking:3} There exists some $i<j$ with $I_i=I_j$ (sequence starts cycling).
\end{enumerate}
Keeping in mind the fact that reductions preserve equalisers (Lemma \ref{lem:MarkedLemma1}), we obtain in each case a subset $\Sigma_{g, h}$ (possibly empty) which forms a basis for $\eq(I_j)$: 
For Case (\ref{Tweaking:1}), writing $\Sigma_j=\{a\}$, the result holds as if $g(a^i)=h(a^i)$ then $g(a)^i=h(a)^i$ and so $g(a)=h(a)$ as roots are unique in a free monoid.
For Case (\ref{Tweaking:2}), suppose ${g}_j(x)={h}_j(x)$. Then ${g_j}$ and ${h_j}$ agree on the first letter of $x \in \Sigma_j^*$ because the image of each letter has length one, and inductively we see that they agree on every letter of $x$. Hence, a subset $\Sigma_{g, h}$ of $\Sigma_j$ forms a basis for $\eq(I_j)$.

For Case (\ref{Tweaking:3}), suppose there is a sequence of reductions beginning and ending at $I_j$:
\[
I_j\rightarrow I_{j+1}\rightarrow\cdots\rightarrow I_{j+(i-1)}\rightarrow I_{j+i}=I_j
\]
and write $r:=j+i$.
By Lemma \ref{lem:MarkedLemma1}, $\eq(I_j)=g_{j+1}g_{j+2}\ldots g_r(\eq(I_r))=\eq(I_r)$; thus $\overline{g_r}:=g_{j+1}g_{j+2}\ldots g_r$ restricts to an automorphism of $\eq(I_j)$, so $\overline{g_r}|_{\eq(I_j)}\in\aut(\eq(I_j))$.
The automorphism $\overline{g_r}$ is necessarily length-preserving ($|\overline{g_r}(w)|=|w|$ for all $w\in \eq(I_j)$). Consider $x\in\eq(I_j)=\eq(I_r)$. Then $\overline{g_r}$ maps the letters occurring in $x_r$ to letters and so $g_j(=g_r)$ and $h_j(=h_r)$ map the letters occuring in $x$ to letters, and it follows that every letter occuring in $x$ is a solution to $I_r=I_j$. Hence, a subset $\Sigma_{g, h}$ of $\Sigma_j$ forms a basis for $\eq(I_j)$ as required.

Therefore, in all three cases a subset $\Sigma_{g, h}$ of $\Sigma_j$ forms a basis for $\eq(I_j)$, and since $\Sigma_j$ is computable, this basis is as well.
In order to prove the theorem, it is sufficient to prove that there is a computable immersion $\psi_{g, h}:\Sigma_{g, h}^*\rightarrow\Sigma^*$. Consider the map $\tilde{g}=g_1g_2\cdots g_j: \Sigma_j^*\rightarrow\Sigma^*$ (and the analogous $\tilde{h}$). Now, each $g_i$ is marked, by Lemma \ref{lem:unfoldableretractsMonoids}, and so $\tilde{g}$ is the composition of marked morphisms and hence is marked itself.
Define $\psi_{g, h}:=\tilde{g}|_{\Sigma_{g, h}^*}$. This map is computable from $\tilde{g}$, and as $\Sigma_{g, h}\subseteq \Sigma_j$, the map $\psi_{g, h}$ is marked. As $\im(\psi_{g, h})=g_{1}g_{2}\ldots g_j(\eq(I_j))=\eq(I)$, by Lemma \ref{lem:MarkedLemma1} and the above, the result follows.
\end{proof}

Theorem \ref{thm:Tweaking} combines with the following general result to give the non-algorithmic part of Theorem \ref{thm:mainMonoids}. A {subsemigroup} $M$ of a free monoid $\Sigma^*$ is \emph{marked} if it is the image of a marked morphism.

\begin{lemma}
\label{lem:mainMonoidsINFINITE}
If $\{M_j\}_{j\in J}$ is a set of marked subsemigroups of $\Sigma^*$ then the intersection $\bigcap_{j\in J}M_j$ is marked.
\end{lemma}

\begin{proof}
Firstly, suppose $x, y\in M_j$ for some $j\in J$. Then there exist two words $x_0\dots x_l$ and $y_0\dots y_k$, with $x_i, y_i \in \Sigma$, such that $\phi(x_0\dots x_l)=x$ and $\phi(y_0\dots y_k)=y$, where $\phi$ is a marked morphism. If $x$ and $y$ have a nontrivial common prefix, then because $\phi$ is marked we get $x_0=y_0$, and $\phi(x_0)$ is a prefix of both $x$ and $y$, and in particular $\phi(x_0) \in M_j$. By continuing this argument, if $z$ is a maximal common prefix of $x$ and $y$, then $z\in M_j$.

Now, suppose $x, y\in\bigcap_{j\in J}M_j$, and suppose they both begin with some letter $a\in \Sigma\cup\Sigma^{-1}$. By the above, their maximal common prefix $z_a$ is contained in each $M_j$ and so is contained in $\bigcap_{j\in J}M_j$. Therefore, $z_a$ is a prefix of every element of $\bigcap_{j\in J}M_j$ beginning with an $a$. It follows that $\bigcap_{j\in J}M_j$ is immersed, as required.
\end{proof}

We now prove the algorithmic part of Theorem \ref{thm:mainMonoids} (this is independent of Lemma \ref{lem:mainMonoidsINFINITE}).

\begin{lemma}
\label{lem:mainMonoidsFINITE}
There exists an algorithm with input a finite set of marked morphisms $S$ from $\Sigma^*$ to $\Delta^*$ and output a marked morphism $\psi_S:\Sigma_S^*\rightarrow \Sigma^*$ such that $\im(\psi_S)=\eq(S)$.
\end{lemma}

\begin{proof}
We use induction on $|S|$.
By Theorem \ref{thm:Tweaking}, the result holds if $|S|=2$.
Suppose the result holds for all sets of $n$ marked morphisms, $n\geq2$, and let $S$ be a set of $n+1$ marked morphisms.
Take elements $g, h\in S$, and write $S_g=S\setminus\{g\}$. By hypothesis, we can algorithmically obtain marked morphisms $\psi_{S_g}:\Sigma_{S_g}^*\rightarrow \Sigma^*$ and $\psi_{g, h}:\Sigma_{g, h}^*\rightarrow \Sigma^*$ such that $\im(\psi_{S_g})=\eq(S_g)$ and $\im(\psi_{g, h})=\eq(g, h)$.

By Lemma \ref{lem:unfoldableretractsMonoids}, there exists a (computable) marked morphism $\psi_S: \Sigma_S^*\rightarrow \Sigma^*$ such that $\im(\psi_S)=\im(\psi_{S_g})\cap\im(\psi_{g, h})$ (the map $\psi_S$ corresponds to the map $k$ in Lemma \ref{lem:unfoldableretractsMonoids}, and $\Sigma_S$ to $\Sigma'$). Then, as required:
$ \im(\psi_S) =\im(\psi_{S_g})\cap\im(\psi_{g, h}) =\eq(S_g)\cap\eq(g, h)=\eq(S).$
\end{proof}

We now prove Theorem \ref{thm:mainMonoids}, which states that the equaliser is the image of a marked map.
\begin{proof}[Proof of Theorem \ref{thm:mainMonoids}]
By applying Lemma \ref{lem:mainMonoidsINFINITE} to Theorem \ref{thm:Tweaking}, there exists an alphabet $\Sigma_S$ and a marked morphism $\psi_S:\Sigma_S^*\rightarrow \Sigma^*$ such that $\im(\psi_S)=\eq(S)$, while by Lemma \ref{lem:mainMonoidsFINITE} if $S$ is finite then such a marked morphism can be algorithmically found.
\end{proof}

We now prove Corollary \ref{corol:rankMonoid}, which says that $\eq(S)$ is free of rank $\leq|\Sigma|$.

\begin{proof}[Proof of Corollary \ref{corol:rankMonoid}]
Consider the marked morphism $\psi_S:\Sigma_S^*\rightarrow\Sigma^*$ given by Theorem \ref{thm:mainMonoids}. By Lemma \ref{lem:MarkedInj}, $\psi_S$ is injective so $\im(\psi_S)$ is free. As $\psi_S$ is marked the map $\Sigma_S\rightarrow\Sigma$ taking each $a\in \Sigma_S$ to the initial letter of $\psi_S(a)$ is an injection, so $|\Sigma_S|\leq|\Sigma|$ as required.
\end{proof}

We now prove a strong form of the \EPM{} for marked morphisms.

\begin{corollary}
\label{corol:basisMonoidSETS}
There exists an algorithm with input a finite set $S$ of marked morphisms from $\Sigma^*$ to $\Delta^*$  and output a basis for $\eq(S)$.
\end{corollary}

\begin{proof}
To algorithmically obtain a basis for $\eq\left({S}\right)$, first use the algorithm of Theorem \ref{thm:mainMonoids} to obtain the marked morphism $\overline{\psi}_{{S}}: \Sigma_{{S}}^*\rightarrow \Sigma^*$ such that $\im(\psi_S)=\eq(S)$. Then, recalling that we store ${\psi}_{{S}}$ as a list $({\psi}_{{S}}(a))_{a\in \Sigma}$, the required basis is the set of elements in this list, so the set $\{{\psi}_{{S}}(a)\}_{a\in \Sigma}$.
\end{proof}

Corollary \ref{corol:basisMonoid}, the \EPM{} for marked morphisms, follows from Corollary \ref{corol:basisMonoidSETS} by taking $|S|=2$, while Corollary \ref{corol:SPCPMonoid}, the simultaneous \PCP{}, also follows as $\eq\left({S}\right)$ is trivial if and only if its basis is empty.

\section{Immersions of free groups}
\label{sec:immersions}

We denote the free group with finite generating set $\Sigma$ by $F(\Sigma)$, and view it as the set of all \emph{freely reduced words} over $\Sigma^{\pm1}=\Sigma \cup \Sigma^{-1}$, that is, words not containing $xx^{-1}$ as subwords, $x\in \Sigma^{\pm1}$, together with the operations of concatenation and free reduction (that is, the removal of any $xx^{-1}$ that might occur when concatenating two words).

We now begin our study of immersions of free groups, as defined in the introduction. 
We first state the characterising lemma, then explain the terms involved before giving the proof.

\begin{lemma}\label{lem:immersionDef}
Let $g: F(\Sigma)\rightarrow F(\Delta)$ be a free group morphism. The following are equivalent.
\begin{enumerate}
\item\label{list:marked} The map $g$ is an immersion of free groups.
\item\label{list:unfoldable} Every word in the language $L(\Gamma_g, v_g)$ is freely reduced.
\item\label{list:kapovichImmersion} For all $x, y\in \Sigma\cup\Sigma^{-1}$ such that $xy\neq1$, the length identity $|g(xy)|=|g(x)|+|g(y)|$ holds.
\end{enumerate}
\end{lemma}

Characterisation (\ref{list:kapovichImmersion}) is the established definition of Kapovich \cite{Kapovich2000Mapping}. Characterisation (\ref{list:unfoldable}) is the one we shall work with in this article. It uses ``Stallings graphs'', which are essentially finite state automata that recognise the elements of finitely generated subgroups of free groups. We define these now, and refer the reader to \cite{Kapovich2002Stallings} for background on Stallings graphs. 

The \emph{(unfolded) Stallings graph} $\Gamma_g$ of the free group morphism $g$ is the directed graph formed by taking a bouquet with $|\Sigma|$ petals attached at a central vertex we call $v_g$, where each petal consists of a path labeled by $g(x) \in (\Delta\cup \Delta^{-1})^*$;  the elements of $\Delta^{-1}$ occur as edges traversed backwards and we denote by $e^{-1}$ the edge $e$ in opposite direction, and by $E\Gamma_g^{\pm1}$ the sets of edges in both directions. A {path} $q=(e_1, \ldots, e_n)$, $e_i\in E\Gamma_g^{\pm1}$ edges, is \emph{reduced} if it has no backtracking, that is, $e_i^{-1}\neq e_{i+1}$ for all $1\leq i<n$. We denote by $\iota(p)$ the initial vertex of a path $p$ and $\tau(p)$ for the terminal vertex, and call a reduced path $p$ with $\iota(p)=u=\tau(p)$ a \emph{closed reduced circuit}.

We shall view $\Gamma_g$ as a finite state automaton $(\Gamma_g, v_g)$ with start and accept states both equal to $v_g$. Then the \emph{extended language accepted by $(\Gamma_g, v_g)$} is the set of words labelling reduced closed circuits at $v_g$ in $\Gamma_g$:
\[
L(\Gamma_g, v_g)=\{label(p)\mid \text{$p$ is a reduced path with $\iota(p)=u=\tau(p)$}\}.
\]
Immersions are precisely those maps $g$ such that every element of $L(\Gamma_g, v_g)$ is freely reduced; this corresponds to the automaton $(\Gamma_g, v_g)$ and the ``reversed'' automaton $(\Gamma_g, v_g)^{-1}$, where edge directions are reversed, both being deterministic (map $g$ in Figure \ref{fig:notImmersion} is not an immersion; although the automaton $(\Gamma_g, v_g)$ is deterministic, $(\Gamma_g, v_g)^{-1}$ is not). For such maps, $L(\Gamma_g, v_g)$ is precisely the image of the map $g$ \cite[Proposition 3.8]{Kapovich2002Stallings}.
\begin{proof}[Proof of Lemma \ref{lem:immersionDef}]
(\ref{list:marked}) $\Leftrightarrow$ (\ref{list:unfoldable}).
Every element of $L(\Gamma_g, v_g)$ is freely reduced if and only if the graph ${\Gamma}_g$, with base vertex $v_{g}$, is such that for all $e_1, e_2\in (E{\Gamma}_g)^{\pm1}$ such that both edes start at $v_g$ or both edges end at $v_g$, then $e_1$ and $e_2$ have different labels (so $\gamma_g(e_1)\neq\gamma_g(e_2)$). This condition on labels is equivalent to $g(\Sigma\cup\Sigma^{-1})$ being marked, as required.

(\ref{list:marked}) $\Leftrightarrow$ (\ref{list:kapovichImmersion}). Condition (\ref{list:kapovichImmersion}) is equivalent to the condition that for all $x, y\in \Sigma\cup\Sigma^{-1}$ such that $xy\neq1$, free cancellation does not happen between $g(x)$ and $g(y)$, which in turn is equivalent to the condition that for all such $x, y$ the elements $g(x^{-1})$ and $g(y)$ start with different letters of $\Delta\cup\Delta^{-1}$. This is equivalent to $g(\Sigma\cup\Sigma^{-1})$ being marked, as required.
\end{proof}

\begin{example}
Let $g: F(a, b)\rightarrow F(x, y)$ be the map defined by $g(a)=x^{-2}y$ and $g(b)=y^2x$. Then the graph $\Gamma_g$, where the double arrow represents $x$ and the single arrow $y$, is depicted in Figure \ref{fig:notImmersion}. The map $g$ is not an immersion since there are two edges labeled $x$ entering $v_g$ (violating Characterisation (\ref{list:unfoldable})). Similarly, $g(a)$ and $g(b^{-1})$ both start with $x^{-1}$ (violating Characterisation (\ref{list:marked})) and $|g(ba)|=4<6=|g(a)|+|g(b)|$ (violating Characterisation (\ref{list:kapovichImmersion})).
\end{example}
\vspace{-1.1cm}
\begin{figure}[h]
\centering
\begin{tikzpicture}
\begin{scope}[decoration={markings,
mark=at position 10/60 with {\arrow[line width=1pt]{stealth[reversed]}},
mark=at position 12/60 with {\arrow[line width=1pt]{stealth[reversed]}},
mark=at position 31/60 with {\arrow[line width=1pt]{stealth[reversed]}},
mark=at position 33/60 with {\arrow[line width=1pt]{stealth[reversed]}},
mark=at position 5/6 with {\arrow[line width=1pt]{stealth}},
}]
\draw[postaction={decorate}] (0, 0) .. controls (-2, 2) and (-2, -2) .. (0, 0);
\end{scope}
\begin{scope}[decoration={markings,
mark=at position 10/60 with {\arrow[line width=1pt]{stealth}},
mark=at position 30/60 with {\arrow[line width=1pt]{stealth}},
mark=at position 5/6 with {\arrow[line width=1pt]{stealth}},
mark=at position 52/60 with {\arrow[line width=1pt]{stealth}},
}] 
\draw[postaction={decorate}] (0, 0) .. controls (2, 2) and (2, -2) .. (0, 0);
\end{scope}
\node at (-2.2, 0) {$\Gamma_g = $};
\filldraw [black] (1,0.56) circle (2pt);
\filldraw [black] (1,-0.56) circle (2pt);
\filldraw [black] (-1,0.56) circle (2pt);
\filldraw [black] (-1,-0.56) circle (2pt);
\filldraw [black] (0,0) circle (2pt);
\node at (0, 0.5) {$v_g$};
\node at (-1.5, -0.9) {$g(a)$};
\node at (1.5, -0.9) {$g(b)$};
\end{tikzpicture}
\caption{
The graph $\Gamma_g$ for the map $g: F(a, b)\rightarrow F(x, y)$ defined by $g(a)=x^{-2}y$, $g(b)=y^2x^{-1}$.
}\label{fig:notImmersion}
\end{figure}
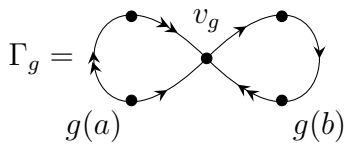

Using Characterisation (\ref{list:unfoldable}), we see that immersions are injective \cite[Proposition 3.8]{Kapovich2002Stallings}:

\begin{lemma}
\label{lem:injective}
If $g: F(\Sigma)\rightarrow F(\Delta)$ is an immersion then it is injective.
\end{lemma}

\section{The reduction of an instance in free groups}
\label{sec:CoreGraph}
By an \emph{immersed} instance of the \PCPG{} we mean an instance $I=(\Sigma, \Delta, g, h)$ where both $g$ and $h$ are immersions.
In this section we define the ``reduction'' of an immersed instance of the \PCPG{}, which is similar to the reduction in the free monoid case.

Let $\Gamma$ be a directed, labeled graph and $u\in V\Gamma$ a vertex of $\Gamma$. 
The \emph{core graph of $\Gamma$ at $u$}, written $\core_u(\Gamma)$, is the maximal subgraph of $\Gamma$ containing $u$ but no vertices of degree $1$, except possibly $u$ itself. 
Note that $L(\core_u(\Gamma), u)=L(\Gamma, u)$.
For $\Gamma_1$, $\Gamma_2$ directed, labeled graphs, the \emph{product graph of $\Gamma_1$ and $\Gamma_2$}, denoted $\Gamma_1\otimes\Gamma_2$, is the subgraph of $\Gamma_1\times\Gamma_2$ with vertex set $V\Gamma_1\times V\Gamma_2$ and edge set $\{(e_1, e_2)\mid e_i\in E\Gamma_i^{\pm1}, label(e_1)=label(e_2)\}$. One may think of the standard construction of an automaton recognising the intersection of two regular languages, each given by a finite state automaton $\Gamma_i$ with start state $s_i$, where the core of $\Gamma_1\otimes\Gamma_2$ at $(s_1,s_2)$ is the automaton recognising this intersection.

\p{Core graph of a pair of morphisms}
Let $g: F(\Sigma_1)\rightarrow F(\Delta)$, $h: F(\Sigma_2)\rightarrow F(\Delta)$ be morphisms. The \emph{core graph of the pair $(g,h)$}, denoted $\core(g, h)$, is the core graph of $\Gamma_g\otimes\Gamma_h$ at the vertex $v_{g, h}=(v_g, v_h)$, so $\core(g, h)=\core_{v_{g, h}}(\Gamma_g\otimes\Gamma_h)$.
We shall refer to $v_{g, h}$ as the \emph{central vertex of $\core(g, h)$}.
Note that $\core(g, h)$ represents the intersection of the two images \cite[Lemma 9.3]{Kapovich2002Stallings}, in the sense that
$$L(\core(g, h), v_{g, h})=\im(g)\cap\im(h).$$
Write $\delta_g:\core(g, h)\rightarrow{\Gamma}_g$ and $\delta_h:\core(g, h)\rightarrow{\Gamma}_h$ for the restriction of $\core(g, h)$ to the $g$ and $h$ components, respectively, so $\delta_g(e_1, e_2)=e_1$, etc.

Now, let $g, h$ be immersions. The graph $\core(g, h)$ is a bouquet and every element of $L(\core(g, h), v_{g, h})$ is freely reduced \cite[Lemma 9.2]{Kapovich2002Stallings}.
We therefore have free group morphisms $g':L(\core(g, h), v_{g, h})\rightarrow L(\Gamma_g, v_g)$ and $h':L(\core(g, h), v_{g, h})\rightarrow L(\Gamma_h, v_h)$ induced by the maps $\delta_g$, $\delta_h$, where $L(\Gamma_g, v_g)=F(\Sigma_1)$ and $L(\Gamma_h, v_h)=F(\Sigma_2)$. These maps are computable \cite[Corollary 9.5]{Kapovich2002Stallings}.
Let $\Sigma'$ be the alphabet whose elements consist of the petals of $\core(g,h)$. Then $\Sigma'$ generates the free group $L(\core(g, h), v_{g, h})$, so $F(\Sigma')=L(\core(g, h), v_{g, h})$, and we see that both $g'$ and $h'$ are immersions with $g': F(\Sigma')\rightarrow F(\Sigma_1)$, $h': F(\Sigma')\rightarrow F(\Sigma_2)$. The map $gg'=hh'$, which we shall call $k$ (so $k: F(\Sigma')\rightarrow F(\Delta)$) is the composition of immersions and hence is itself an immersion.
We therefore have the following.

\begin{lemma}
\label{lem:unfoldableretracts}
If $g:F(\Sigma_1)\rightarrow F(\Delta)$ and $h:F(\Sigma_2)\rightarrow F(\Delta)$ are immersions then the corresponding maps $g': F(\Sigma')\rightarrow F(\Sigma_1)$, $h': F(\Sigma')\rightarrow F(\Sigma_2)$ and $k: F(\Sigma)\rightarrow F(\Delta)$, where $k=gg'=hh'$, are immersions and are computable.
\end{lemma}


\p{Reduction}
The \emph{reduction} of an immersed instance $I=(\Sigma, \Delta, g, h)$ of the \PCPG{} is the instance $I'=(\Sigma', \Delta, g', h')$ where $\Sigma'$ is defined as above, and where $g'$ and $h'$ are as above, but with codomain $\Delta$ (which we may do as $\Sigma\subseteq\Delta$).
We additionally assume that $\Sigma'\subseteq\Sigma$; we can do this as $|\Sigma'|\leq|\Sigma|$ by Lemma \ref{lem:unfoldableretracts}.
As $I$ is immersed, it follows from Lemma \ref{lem:unfoldableretracts} that $I'$ is also immersed.
In the next section we show that the name ``reduction'' makes sense, as it reduces the ``prefix complexity'' of instances.
\begin{example}\label{ex:reduction}
Consider the maps $g, h:F(a, b, c) \rightarrow F(x, y, z)$ given by $g(x)=aba^2, g(b)=y^{-1}, g(c)=zxz$ and $h(a)=x, h(b)=yx^2y, h(c)=z$.

Then the graph $\core(g, h)$ is a bouquet with two petals labelled $xyx^2y$ and $zxz$, and $\im(g)\cap \im(h)=\langle xyx^2y, zxz\rangle$. Moreover, $g(ab^{-1})=h(ab)=xyx^2y$ and $g(z)=h(zxz)=zxz.$ Then we can take $\Sigma'=\{a',b'\}$, and the maps given by $g'(a')=ab^{-1}, g(b')=c$ and $h'(a')=ab, h'(b')=cac$ are the reduction of $(g,h)$.
\end{example}

We now prove that reduction preserves equalisers.
Two instances $I_1$ and $I_2$ of the \PCPG{} are \emph{strongly equivalent} if the equalisers are isomorphic, which we write as $\eq(I_1)\cong\eq(I_2)$.

\begin{lemma}
\label{lem:equivalent}
Let $I'=(\Sigma', \Delta', g', h')$ be the reduction of $I=(\Sigma, \Delta, g, h)$ where $g$ and $h$ are immersions. Then $I$ and $I'$ are strongly equivalent, and $g'(\eq(I'))=\eq(I)=h'(\eq(I'))$.
\end{lemma}

\begin{proof}
It is sufficient to prove that $g'|_{\eq(I')}$ is injective and $g'(\eq(I'))=\eq(I)$; that $h'(\eq(I'))=\eq(I)$ follows as $g'|_{\eq(I')}=h'|_{\eq(I')}$.

As $g'$ is an immersion it is injective, by Lemma \ref{lem:injective}. Therefore, $g'|_{\eq(I')}$ is injective. To see that $\im(g'|_{\eq(I')})\leq\eq(I)$, suppose $x'\in\eq(I')$. Writing $x=g'(x')=h'(x')$, we have
$ g(x)=gg'(x')=hh'(x')=h(x)$ and so $x=g'(x')\in\eq(I)$, as required.

To see that $\im(g'|_{\eq(I')})\geq\eq(I)$, 
suppose $x\in\eq(I)$.
Then there exists a path $p_x\in\core(g, h)$, $\iota(p_x)=v_{g,h}=\tau(p_x)$, such that $\gamma_g\delta_g(p_x)=g(x)=h(x)=\gamma_h\delta_h(p_x)$ \cite[Proposition 9.4]{Kapovich2002Stallings}, where $\gamma_g: \Gamma_g\rightarrow\Gamma_{\Delta}$ is the canonical morphism of directed, labeled graphs from $\Gamma_g$ to the bouquet $\Gamma_{\Delta}$ with $\Delta$ petals.
 Hence, writing $x'$ for the element of $F(\Sigma')$ corresponding to $p_x\in L(\core(g, h), v_{g, h})$, we have that $gg'(x')=g(x)=h(x)=hh'(x)$. As $h$ and $g$ are injective, by Lemma \ref{lem:injective}, we have that $g'(x')=x=h'(x')$ as required.
\end{proof}


\section{Prefix complexity of immersions in free groups}
\label{sec:prefix}
In this section we associate to an instance $I$ of the \PCPG{} a certain complexity, called the ``prefix complexity''. We prove that the process of reduction does not increase this complexity, and that for all $n\in\mathbb{N}$ there are only finitely many instances with complexity $\leq n$.

Let $I=(\Sigma, \Delta, g, h)$ be an immersive instance of the \PCPG{}.
We define, analogously to \cite[Section 4]{Halava2001marked} (see also \cite{Ehrenfeucht1982generalized}), the \emph{prefix complexity} $\sigma(I)$ as:
\begin{align*}
\sigma(I)
&=\left|\cup_{a\in\Sigma^{\pm1}}\{x\in F(\Delta)\mid \text{$x$ is a proper prefix of $g(a)$}\}\right|\\
&+\left|\cup_{a\in\Sigma^{\pm1}}\{x\in F(\Delta)\mid \text{$x$ is a proper prefix of $h(a)$}\}\right|.
\end{align*}
In the maps in Example \ref{ex:reduction}, $\sigma(I)=10+6=16$, and $\sigma(I')=2+4=6.$

The process of reduction does not increase the prefix complexity, and we prove this by using the fact that, for any $a\in \Sigma^{\pm1}$, the proper prefixes of $g(a)$ and $h(a)$ are in bijection with the proper initial subpaths of the petals of $\Gamma_g$ and $\Gamma_h$, respectively.

\begin{lemma}
\label{lem:PrefixComplexity}
Let $I=(\Sigma, \Delta, g, h)$ be an instance of the \PCPG{} with $g$ and $h$ immersions, and let $I'$ be the reduction of $I$. Then $\sigma(I')\leq\sigma(I)$.
\end{lemma}

\begin{proof}
We write $V_g\core(g, h)=\{(v_g, v) \in V\core(g, h) \mid v \in \Gamma_h\}=\delta_g^{-1}(v_g)$ for the set of vertices in the $\core(g,h)$ whose first component is the central vertex $v_g$ of $\Gamma_g$, and similarly for $V_h\core(g, h)$. Note that $V_g\core(g, h)\cap V_h\core(g, h)=\{v_{g, h}\}$.

By construction, each petal of $\Gamma_g$ and $\Gamma_h$ corresponds to a letter $a\in \Sigma^{\pm1}$, and we shall denote the petal also by $a$. Write $P{\Gamma}$ for the set of reduced paths in a graph ${\Gamma}$.
Similarly to $a\in P\Gamma_g$ and $a\in P\Gamma_h$, we map write $a\in P\core(g, h)$ for the petal in $\core(g, h)$ corresponding to $a\in (\Sigma')^{\pm1}$. From now on, all paths are assumed to be reduced.
Define
\begin{align*}
G&=\cup_{a\in\Sigma^{\pm1}}\{p\in {\Gamma}_{g}\mid \text{$p$ is a proper initial subpath of petal $a\in {\Gamma}_{g}$}\},\\
G'&=\cup_{a\in(\Sigma')^{\pm1}}\{p\in \core(g, h)\mid \text{$p$ is a proper initial subpath}\\&~~~~~~~~~~\text{of $a\in \core(g, h)$ s.t. its end vertex $\tau(p)\in V_{g}\core(g, h)$}\},
\end{align*}
and define $H$ and $H'$ analogously.
Hence, $\sigma(I)=|G|+|H|$ and analogously $\sigma(I')=|G'|+|H'|$. 

For $a\in(\Sigma')^{\pm1}$ let $q\in G'$ be a subpath of $a\in P\core(g, h)$. Denote by $r_q$ the shortest subpath of $a$ intersecting $q$ at only one point, their common end vertex (that is, $\tau(q)=\tau(r_q)=q\cap r_q$), such that $\iota(r_q)\in V_{h}\core(g, h)$; the paths $q$ and $r_q$ can be seen as ``facing'' one another on $a$. As $V_{g}\core(g, h)\cap V_{h}\core(g, h)=\{v_{g, h}\}$, and as $q$ is a proper initial subpath of $a$, the projection $\delta_{h}(r_q)$ is a non-trivial path in $\Gamma_h$. Note also that $\iota(\delta_{h}(r_q))=v_{h}$, as $\iota(r_q)\in V_{h}\core(g, h)$, hence there exists some $b\in\Sigma^{\pm1}$ such that $\delta_{h}(r_q)$ is a proper initial subpath of the petal $b\in \Gamma_h$.
Therefore, $\delta_{h}(r_q)\in H$. Let $\xi_H: G'\rightarrow H$ be the map given by $\xi_H(q)=\delta_{h}(r_q)$.

We now prove that $\xi_H$ is injective.
Suppose $p, q\in G'$ are such that $\xi_H(p)=\xi_H(q)$, and let $r_p$ and $r_q$ be the paths obtained from $p$ and $q$, respectively, such that $\xi_H(p)=\delta_{h}(r_p)$ and $\xi_H(q)=\delta_{h}(r_q)$.
Write $e_p$ for the terminal edge of $r_p$, and $e_q$ for the terminal edge of $r_q$, and note that these two edges have the same label and direction as $\delta_{h}(e_p)=\delta_{h}(e_q)$.
Now, $\delta_{g}(\tau(e_p))=v_g=\delta_{g}(\tau(e_q))$ as $\tau(r_p), \tau(r_q)\in V_{g}\core(g, h)$, and as $e_p$ and $e_q$ have the same label and direction we have that $\delta_g(e_p)=\delta_g(e_q)$. Therefore, both $\delta_g$ and $\delta_h$ agree on $e_p$ and $e_q$, and so as $\core(g, h)$ is a subgraph of $\Gamma_g\times\Gamma_h$ we have that $e_p=e_q$. As $\core(g, h)$ is a bouquet, there exists a unique shortest reduced path $s$ such that $\iota(s)=v_{g, h}$ and $\tau(s)=s\cap e_p=\tau(e_p)$. Hence, $p=s=q$ as required.

Thus $\xi_H$ is injective, and so $|G'|\leq |H|$. The same will hold for an analogously defined function $\xi_H$ from $H'$ to $G$, so $|H'|\leq|G|$. Therefore, $\sigma(I')=|G'|+|H'|\leq|G|+|H|=\sigma(I)$.
\end{proof}

For a fixed number $n\geq1$ there are obviously only finitely many words which have $\leq n$ proper prefixes, and so the following is clear:

\begin{lemma}
\label{lem:DistinctInstances}
There exist only finitely many distinct instances $I=(\Sigma, \Delta, g, h)$ of the \PCPG{} that satisfy $\sigma(I)\leq n$.
\end{lemma}

As the reduction $I'$ of an instance $I$ gives $\sigma(I')\leq \sigma(I)$, and as $|\Sigma'|\subseteq|\Sigma|$, this means that the process of iteratively computing reductions will eventually cycle.


\section{Solving the Algorithmic Equaliser Problem in free groups (\EPG{})}
\label{sec:TheoremsForPairs}
The algorithm for solving the \EPG{} for immersions is analogous to the algorithm for marked free monoid morphisms in Section \ref{sec:PCPM}.
Our algorithm starts by making reductions $I_0, I_1, I_2, \ldots$, beginning with $I_0=I$, the input instance. By Lemma \ref{lem:DistinctInstances}, we will obtain an instance $I_j=(\Sigma_j, \Delta, g_j, h_j)$ such that one of the following will occur:
\begin{enumerate}
\item\label{PCPAlg:case1} $|\Sigma_j|=1$.
\item\label{PCPAlg:case2} $\sigma(I_j)=0$.
\item\label{PCPAlg:case3} there exists some $i<j$ with $I_i=I_j$ (sequence starts cycling).
\end{enumerate}
Keeping in mind the fact that reductions preserve equalisers (Lemma \ref{lem:equivalent}), we obtain in each case a subset $\Sigma_{g, h}$ (possibly empty) which forms a basis for $\eq(I_j)$: For Case (\ref{PCPAlg:case1}), writing $\Sigma_j=\{a\}$, the result holds as if $g(a^i)=h(a^i)$ then $g(a)^i=h(a)^i$ and so $g(a)=h(a)$ as roots are unique in a free group. For Case (\ref{PCPAlg:case2}), $\sigma(I_j)=0$ is equivalent to $|g(a)|=|h(a)|=1$ for all $a\in \Sigma$. Suppose there exists some non-trivial reduced word $x=a_{i_1}^{\epsilon_1}\cdots a_{i_n}^{\epsilon_n}$ such that $g(x)=h(x)$. Then as $g$ and $h$ are injective, the words $g(a_{i_1})^{\epsilon_1}\cdots g(a_{i_n})^{\epsilon_n}$ and $h(a_{i_1})^{\epsilon_1}\cdots h(a_{i_n})^{\epsilon_n}$ are freely reduced and hence are the same word, and so $g(a_{i_j})=h(a_{i_j})$. The result then follows for Case (\ref{PCPAlg:case2}).
Case (\ref{PCPAlg:case3}) has a more involved proof.

\begin{lemma}
\label{lem:BasisOfReduction3}
Let $I=(\Sigma, \Delta, g, h)$ be an immersive instance of the \PCPG{} that starts a cycle (i.e. starting the reduction process with $I$ eventually gives $I$ again).
If $\eq(I)$ is non-trivial then a subset of $\Sigma$ forms a basis for $\eq(I)$.
\end{lemma}

\begin{proof}
There is a sequence of reductions beginning and ending at $I$:
\[
I=I_0\rightarrow I_1\rightarrow \cdots\rightarrow I_{r-1}\rightarrow I_r=I
\]
where $I_i=(\Sigma_i, \Delta, g_i, h_i)$. By Lemma \ref{lem:equivalent}, $\eq(I_0)=g_1g_2\ldots g_r(\eq(I_r))=\eq(I_r)$ and so $\overline{g_r}=g_1g_2\ldots g_r$ restricts to an automorphism of $\eq(I_0)$, that is, $\overline{g_r}|_{\eq(I_0)}\in\aut(\eq(I_0))$. For $\overline{h_r}$ defined analogously, $\overline{h_r}|_{\eq(I_0)}\in\aut(\eq(I_0))$.
Write $\eq(I_k)^{(n)}$ for the set of words in $\eq(I_k)$ of length precisely $n$, and $\eq(I_k)^{(\leq n)}$ for the set of words in $\eq(I_k)$ of length at most $n$. Consider some $x_0\in\eq(I_0)$ and write $x_r=\overline{g_r}^{-1}(x_0)$. Then
\begin{align*}
x_0&=g_1g_2\ldots g_r(x_r)=\overline{g_r}(x_r),\\
x_0&=h_1h_2\ldots h_r(x_r)=\overline{h_r}(x_r).
\end{align*}
By Lemma \ref{lem:unfoldableretracts}, both $g_i$ and $h_i$ are immersions for each $i$, and so by Characterisation (\ref{list:kapovichImmersion}) of Lemma \ref{lem:immersionDef} we see that $|g_i(w)|\geq|w|$ for all $w\in F(\Sigma_i)$. Hence, $|x_0|\geq|x_r|$. Therefore, for all $m\geq1$ the map $\overline{g_r}$ induces a map $\overline{g_r}^{(m)}:\eq(I_r)^{(m)}\rightarrow\eq(I_r)^{(\leq m)}$. Clearly $\overline{g_r}^{(1)}$ is a bijection, and so we inductively see that $\overline{g_r}^{(m)}$ has image $\eq(I_r)^{(m)}$. Therefore, the automorphism $\overline{g_r}$ of $\eq(I_0)$ is length-preserving ($|\overline{g_r}(w)|=|w|$ for all $w\in \eq(I)$), and so maps the letters occurring in $x_r$ to letters.
Hence, $g_0(=g_r)$ and $h_0(=h_r)$ map the letters occuring in $x_r$ to letters, and it follows that every letter occuring in $x_r$ is a solution to $I_0$. Hence, a subset $\Sigma_{g, h}$ of $\Sigma_r$ forms a basis for $\eq(I_r)$.
\end{proof}

We now prove the central theorem of this article, which gives an algorithm to describe $\eq(I)$ as the image of an immersion. Note that not every subgroup of a free group is the image of an immersion: for example, if $|\Sigma|=n$, then no subgroup of $F(\Sigma)$ of rank $>n$ is the image of an immersion.
We store a morphism $f:F(\Sigma)\rightarrow F(\Delta)$ as a list $(f(a))_{a\in \Sigma}$.

\begin{theorem}
\label{thm:main1}
There exist an algorithm with input an immersive instance $I=(\Sigma, \Delta, g, h)$ of the \PCPG{} and output an immersion $\psi_{g, h}:F(\Sigma_{g, h})\rightarrow F(\Sigma)$ such that $\im(\psi_{g, h})=\eq(I)$.
\end{theorem}

\begin{proof}
Start by making reductions $I=I_0\rightarrow I_1\rightarrow \cdots$. By Lemma \ref{lem:DistinctInstances} we will obtain an instance $I_j=(\Sigma_j, \Delta, g_j, h_j)$ satisfying one of the Cases (\ref{PCPAlg:case1})--(\ref{PCPAlg:case3}) above, and in each case a subset $\Sigma_{g, h}$ of $\Sigma_j$ forms a basis for $\eq(I_j)$.
Since $\Sigma_j$ is computable, this basis is as well.

In order to prove the theorem, it is sufficient to prove that there is a computable immersion $\psi_{g, h}:F(\Sigma_{g, h})\rightarrow F(\Sigma)$. Consider the map $\tilde{g}=g_1g_2\cdots g_j: F(\Sigma_j)\rightarrow F(\Sigma)$ (and the analogous $\tilde{h}$). Now, each $g_i$ is an immersion, so $\tilde{g}$ is the composition of immersions and hence is an immersion.
Define $\psi_{g, h}=\tilde{g}|_{F(\Sigma_{g, h})}$. This map is computable from $\tilde{g}$, and as $\Sigma_{g, h}\subseteq \Sigma_j$, the map $\psi_{g, h}$ is an immersion. As $\im(\psi_{g, h})=g_{1}g_{2}\ldots g_j(\eq(I_j))=\eq(I)$, by Lemma \ref{lem:equivalent} and the above, the result follows.
\end{proof}

We now prove Corollary \ref{thm:basis}, which solves the \EPG{} for immersions of free groups.

\begin{proof}[Proof of Corollary \ref{thm:basis}]
To algorithmically obtain a basis for $\eq(I)$, first obtain the immersion $\psi_{g, h}: F(\Sigma_{g, h})\rightarrow F(\Sigma)$ given by Theorem \ref{thm:main1}. Then, recalling that we store $\psi_{g, h}$ as a list $(\psi_{g, h}(a))_{a\in \Sigma}$, the required basis is the set of elements in this list, so the set $\{\psi_{g, h}(a)\}_{a\in \Sigma}$.
\end{proof}


\section{Sets of immersions}
\label{sec:MainResults}
We now prove Theorem \ref{thm:main} and its corollaries. We first give a general result, from which the non-algorithmic part of Theorem \ref{thm:main} follows quickly. An \emph{immersed subgroup} $H$ of a free group $F(\Sigma)$ is a subgroup which is the image of an immersion. The proof of Lemma \ref{lem:mainINFINITE} is fundamentally identical to the proof of Lemma \ref{lem:mainMonoidsINFINITE}, via Characterisation \ref{list:marked} of Lemma \ref{lem:immersionDef}.

\begin{lemma}
\label{lem:mainINFINITE}
If $\{H_j\}_{j\in J}$ is a set of immersed subgroups of $F(\Sigma)$ then the intersection $\bigcap_{j\in J}H_j$ is immersed.
\end{lemma}

\begin{proof}
Firstly, suppose $x, y\in H_j$ for some $j\in J$, and let $z$ be their maximal common prefix. Then $z$ decomposes uniquely as $z_1z_2\cdots z_nz_{n+1}'$ such that each $z_k\in H_j$. As $H_j$ is immersed, and as $z$ is a maximal common prefix of $x$ and $y$, we have that $z\in H_j$.

Now, suppose $x, y\in\bigcap_{j\in J}H_j$, and suppose they both begin with some letter $a\in \Sigma\cup\Sigma^{-1}$. By the above, their maximal common prefix $z_a$ is contained in each $H_j$ and so is contained in $\bigcap_{j\in J}H_j$. Therefore, $z_a$ is a prefix of every element of $\bigcap_{j\in J}H_j$ beginning with an $a$. It follows that $\bigcap_{j\in J}H_j$ is immersed, as required.
\end{proof}

The following lemma corresponds to the algorithmic part of Theorem \ref{thm:main}. Similar to the above, the proof of the lemma is fundamentally identical to the proof of Lemma \ref{lem:mainMonoidsFINITE}.

\begin{lemma}
\label{lem:mainFINITE}
There exists an algorithm with input a finite set of immersions $S$ from $F(\Sigma)$ to $F(\Delta)$ and output an immersion $\psi_S:F(\Sigma_S)\rightarrow F(\Sigma)$ such that $\im(\psi_S)=\eq(S)$.
\end{lemma}
\begin{proof}
We proceed by inducting on $|S|$.
By Theorem \ref{thm:main1}, the result holds if $|S|=2$.
Suppose the result holds for all sets of $n$ immersions, $n\geq2$, and let $S$ be a set of $n+1$ immersions.
Take elements $g, h\in S$, and write $S_g=S\setminus\{g\}$. By hypothesis, we can algorithmically obtain immersions $\psi_{S_g}:F(\Sigma_{S_g})\rightarrow F(\Sigma)$ and $\psi_{g, h}:F(\Sigma_{g, h})\rightarrow F(\Sigma)$ such that $\im(\psi_{S_g})=\eq(S_g)$ and $\im(\psi_{g, h})=\eq(g, h)$.

By Lemma \ref{lem:unfoldableretracts}, there exists a (computable) immersion $\psi_S: F(\Sigma_S)\rightarrow F(\Sigma)$ such that $\im(\psi_S)=\im(\psi_{S_g})\cap\im(\psi_{g, h})$ (the map $\psi_S$ corresponds to the map $k$ in the lemma, and $\Sigma_S$ to $\Sigma'$). Then we have the required equality:
\begin{align*}
\im(\psi_S)
&=\im(\psi_{S_g})\cap\im(\psi_{g, h})\\
&=\eq(S_g)\cap\eq(g, h)\\
&=\eq(S).
\end{align*}
\end{proof}

We now prove Theorem \ref{thm:main}, which states that the equaliser is the image of a computable immersion.

\begin{proof}[Proof of Theorem \ref{thm:main}]
By Lemma \ref{lem:mainINFINITE}, there exists an alphabet $\Sigma_S$ and an immersion $\psi_S:F(\Sigma_S)\rightarrow F(\Sigma)$ such that $\im(\psi_S)=\eq(S)$, while by Lemma \ref{lem:mainFINITE} if $S$ is finite then such an immersion can be algorithmically found.
\end{proof}

We now prove Corollary \ref{corol:SPCP}, which solves the simultaneous \PCPG{} for immersions.

\begin{proof}[Proof of Corollary \ref{corol:SPCP}]
First find a basis for $\eq(I)$: obtain the immersion $\psi_{g, h}: F(\Sigma_{g, h})\rightarrow F(\Sigma)$ given by Theorem \ref{thm:main}. Then, recalling that we store $\psi_{g, h}$ as a list $(\psi_{g, h}(a))_{a\in \Sigma}$, the required basis is the set of elements in this list, so the set $\{\psi_{g, h}(a)\}_{a\in \Sigma}$.
Then $\eq(S)$ is trivial if and only if this basis is empty.
\end{proof}

Finally, we prove Corollary \ref{thm:rankSETS}, which says that $\eq(S)$ is of rank $\leq|\Sigma|$.

\begin{proof}[Proof of Corollary \ref{thm:rankSETS}]
Consider the immersion $\psi_{g, h}: F(\Sigma_{g, h})\rightarrow F(\Sigma)$ given by Theorem \ref{thm:main}.
As $\im(\psi_{g, h})=\eq(S)$ we have that $\rk(\eq(S))\leq|\Sigma_{g, h}|$, while as $\psi_{g, h}$ is an immersion we have that $|\Sigma_{g, h}|\leq|\Sigma|$, and the result follows.
\end{proof}


\section{Algorithm to compute the equaliser}
\label{sec:algorith}
Theorems \ref{thm:mainMonoids} and \ref{thm:main} produce the equaliser of a set $S$ of morphisms as the image of a computable map $\psi_S$.  For $S=\{g, h\}$, the structure of the algorithm that gives $\psi_S$ (as a list of elements representing the images of the generators) is given below. 
The values for $M$ in step \ref{pairs:2} correspond to the number of instances of complexity $\leq\sigma(I)$, as explained in Section \ref{sec:Complexity}.
%
\begin{enumerate}
\item Input $I=(\Sigma, \Delta,g,h)$. 
\item Set $c=;0$, $i:=0$, $I_0:=I$
\item\label{pairs:2} Set $M:=(|\Delta|+1)^{2|\Sigma|(\sigma(I)+1)}$ (monoids) or $M:=(2|\Delta|)^{2|\Sigma|(\sigma(I)+1)}$ (groups)
\item\label{pairs:3}  $i:=i+1$
\item Reduce instance $I_{i-1}$ to $I_i$ (as in Sections \ref{sec:PCPM} and \ref{sec:CoreGraph}); store $I_i$ in memory
\item If $I_i$ has source alphabet of size $1$ or $\sigma=0$ then:
\begin{enumerate}
\item Compute a basis $B$ for $\eq(I_i)$
\item Print \texttt{composition($B$, $i$)} (see below) and terminate.
\end{enumerate}
\item If $I_i$ is simpler than $I_{i-1}$ (smaller source alphabet or $\sigma$) then set $c=0$ and goto (\ref{pairs:3})
\item If $c>M$ then there exists a cycle which starts with $I_i$.
\begin{enumerate}
\item Compute a basis $B$ for $\eq(I_i)$
\item Print \texttt{composition($B$, $i$)} and terminate.
\end{enumerate}
\end{enumerate}

Procedure {\texttt{composition($B$, $i$)} computes the composition of a map, stored as a list $B$, with the maps obtained in the reduction process, indexed from $i$ downwards.

\vspace{5pt}
\noindent{\texttt{composition($B$, $i$)}}
\begin{enumerate}
\item\label{Composition:1} Set $B:=g_i(B)$, where $g_i$ is loaded from memory
\item  $i:=i-1$
\item If $i\geq0$, goto (\ref{Composition:1}); else, output $B$.
\end{enumerate}



\section{Complexity analysis}
\label{sec:Complexity}
The size of an instance $I=(\Sigma, \Delta, S)$, $S$ a set of morphisms, is $|\Sigma|+|\Delta|+\sum_{g\in S}\sum_{a \in \Sigma^{\pm 1}}|g(a)|$.
The algorithm underlying Theorem \ref{thm:mainMonoids} can be run with $O(2^n)$ space, where $n$ is the size of the input instance $I$, which gives a time bound of $O(2^{2^n})$. The space grows exponentially, unlike in \cite{Halava2001marked}, because the algorithm computes instances that must each be stored (as the immersion $\psi_{g, h}$ is their composition; this corresponds to the function \texttt{composition($B$, $i$)}, above).
To obtain this space complexity, first suppose $|S|=2$ (so consider the function \texttt{pairs($g$, $h$)}, above). There are at most $(|\Delta|+1)^{2|\Sigma|(\sigma(I)+1)}$ instances $I_j$ with $\sigma(I_j)\leq\sigma(I)$ \cite[Proof of Lemma 3]{Halava2001marked}, which is $O(2^n)$. Every other procedure requires asymptotically less space, and hence if $|S|=2$ we require $O(2^n)$ space. For $S=\{g_1, \ldots, g_k\}$, note that we only need to compute the immersions corresponding to $\eq(g_{i}, g_{i+1})$ for $1\leq i<k$ (as these intersect to give $\eq(S)$), and these can all be stored in $(k-1)\times O(2^n)=O(2^n)$ space. Intersection corresponds to reduction, and reduction can be done in PSPACE \cite[Section 6]{Halava2001marked}. Hence, the algorithm can be run in $O(2^n)$ space.

Similarly, the algorithm underlying Theorem \ref{thm:main} runs in $O(2^n)$ space, where $n$ is the input size. The main difference to the above is that there are $O(2^n)$ instances $I_j$ with $\sigma(I_j)\leq\sigma(I)$. To see this, write $m:=\sigma(I)$ and $d:=|\Delta|$. If $I_j=(\Sigma_j, \Delta_j, g_j, h_j)$ is such that $\sigma(I_j)\leq m$ then $|g(a)|\leq m+1$ for all $a\in\Sigma_j^{\pm1}$, as $g(a)$ has at most $m$ proper prefixes, and similarly $|h(a)|\leq m+1$. There are $2d(2d-1)^{m}$ freely reduced words of length $m+1$ in $F(\Sigma_j)$, and so (by using the empty word) we see that there are at most $(2d)^{m+1}$ freely reduced words of length \emph{at most} $m+1$.
As each list of $2|\Sigma_j|$ words defines an instance, there are at most $\left(2d\right)^{2|\Sigma_j|(m+1)}\leq\left(2d\right)^{2|\Sigma|(m+1)}$ instances that satisfy $\sigma(I)\leq m$. This is $O(2^n)$ as required.

\section{The density of marked morphisms and immersions} \label{sec:Density}
Here we show that immersions and marked morphisms are not a negligible (i.e. density zero) subset of the entire set of free group and free monoid morphisms, respectively, but represent a strictly positive proportion of those.

Suppose $\Sigma=\{a_1, \dots, a_k\}$, and $k=|\Sigma| \geq |\Delta|=m$. 
A morphism in a free monoid or free group, $\phi:\Sigma^* \rightarrow \Delta^*$ or $\phi:F(\Sigma) \rightarrow F(\Delta)$, is uniquely determined by $(\phi(a_1), \dots,\phi(a_k))$. 

We start with the monoid case. There are $m^n$ words of length $n$ in $\Delta^*$, and $\sum_{1 \leq i \leq n} m^i \sim c m^n$ words of length $\leq n$, where $c=\frac{m}{m-1}$ and we write $a_n \sim b_n$ for $\lim_{n\to \infty}\frac{a_n}{b_n}=1$. If $\alpha_n$ is the number of morphisms from $\Sigma^*$ to $\Delta^*$ with images of length at most $n$, then $\alpha_n \sim (cm^n)^k$. Now let $\beta_n$ be the number of marked morphisms from $\Sigma^*$ to $\Delta^*$ with images of length at most $n$. For a marked morphism $\phi$, each word in the list $(\phi(a_1), \dots,\phi(a_k))$ must start with a different letter,  followed by any word of length $\leq n-1$. Since there are ${m \choose k}k!$ options for the first letters, $\beta_n \sim {m \choose k} k!(cm^{n-1})^k$ and we get:
\begin{proposition}
If $\alpha_n$ and $\beta_n$ are the numbers of morphisms and marked morphisms, respectively, from $\Sigma^*$ to $\Delta^*$, with images of length at most $n$, then the density of the marked morphisms among all morphisms is a positive constant: 
$$
\lim_{n\to \infty}\frac{\beta_n}{\alpha_n}=\lim_{n\to \infty}\frac{{m \choose k}k! (cm^{n-1})^k}{(cm^n)^k} = \frac{m!}{m^k(m-k)!}.
$$
\end{proposition}

In the free group case the counting is similar, but there are more restrictions on the images of an immersion: first, all images need to be reduced words, and second, not just their first letters are constrained, but also their last letters. For some $\phi$, let the set of first letters of $(\phi(a_1), \dots,\phi(a_k))$ be $F \subset \Delta^{\pm 1}$, and the set of inverses of the last letters be $L\subset \Delta^{\pm 1}$. Then $\phi:F(\Sigma) \mapsto F(\Delta)$ is an immersion if all letters in $F$ are distinct, all the letters in $L$ are distinct, which implies $|F|=|L|=k$, and furthermore $F \cap L =\emptyset$. An image $\phi(a_i)$ of length $n$ has the form $\phi(a_i)=\alpha x_1 x_2 \dots x_{n-2} \beta$, where $\alpha \in F$, $\beta^{-1} \in L$, $x_i \in \Delta^{\pm 1}$, and $\phi(a_i)$ is reduced, so $x_1 \neq \alpha^{-1}$ and $x_{n-2}\neq \beta^{-1}$. 
Counting such words is more delicate than in the monoid case, but the asymptotics are similar, due to the following result (\cite[Proposition 1]{EJC2006}).

\begin{proposition}
Let $A$ and $B$ be subsets of $\Delta^{\pm 1}$. The number of elements of length $n$ in $F(\Delta)$ that do not start with a letter in $A$ and do not end with a letter in $B$ is equal to
$$
f_{A,B}(n)=\frac{(2m - |A|)(2m - |B|)(2m - 1)^{n-1} + xm + (-1)^n(|A||B| - ym)}{2m},
$$
where $x = |A \cap B| - |A^{-1} \cap B|, y = |A \cap B| + |A^{-1} \cap B|$, and $m=|\Delta|$.
\end{proposition} 
Let $A=\{\alpha^{-1}\}$ and $B=\{\beta^{-1}\}$; then since the number of possible $\phi(a_i)$ of length $\leq n$ is equal to the number of reduced words of length $\leq n-2$ starting with a letter different from $\alpha^{-1}$ and ending with a letter different from $\beta^{-1}$, this number is $\sum_{1\leq j \leq n-2} f_{A,B}(j)$. Since $|A|=|B|=1$, $f_{A,B}(j)$ is asymptotically $(2m-1)^j$, and the number of possible $\phi(a_i)$ is $\sim c_1 (2m-1)^{n-2}$, where $c_1$ is a constant depending on $m$. 
Thus for fixed sets $F$ and $L$ the number of immersions $\phi$ with images in the ball of radius $n$ is $\sim (c_1(2m-1)^{n-2})^k $. Since there are only finitely many choices for sets $F$ and $L$ of first and last letters, respectively, and the number of $k$-tuples of elements in $F(\Delta)$ of length $\leq n$ is $\sim (c_2(2m-1)^n)^k$ for some constant $c_2$, the number of immersions over the total number of maps $F(\Sigma) \mapsto F(\Delta)$ is $\sim \frac{(c_1(2m-1)^{n-2})^k}{(c_2(2m-1)^n)^k}$; so as $n \mapsto \infty$, this ratio is a positive constant depending on $k$ and $m$.



\bibliographystyle{amsalpha}
\bibliography{BibTexBibliography}

\end{document}